\documentclass[12pt,notitlepage]{amsart}
\usepackage{latexsym,amsfonts,amssymb,amsmath,amsthm}
\usepackage{graphicx}
\usepackage{color}
\usepackage[normalem]{ulem}
\let\turc\c
\usepackage{etoolbox}
\robustify\turc 
\renewcommand{\c}{\mathfrak{c}}
\newcommand{\bpm}{\begin{pmatrix}}
\newcommand{\epm}{\end{pmatrix}}

\usepackage[inner=1.0in,outer=1.0in,bottom=1.0in, top=1.0in]{geometry}

\newcommand{\mz}{\ensuremath{\mathbb Z}}

\newcommand{\mymod}{\ensuremath{\negthickspace \negmedspace \pmod}}
\newcommand{\shortmod}{\ensuremath{\negthickspace \negthickspace \negthickspace \pmod}}

\newcommand{\intR}{\int_{-\infty}^{\infty}}

\newcommand{\sumstar}{\sideset{}{^*}\sum}

\theoremstyle{plain}		
	\newtheorem{mytheo}{Theorem} [section]
	
	\newtheorem{myprop}[mytheo]{Proposition}
	\newtheorem{mycoro}[mytheo]{Corollary}
     \newtheorem{mylemma}[mytheo]{Lemma}
	\newtheorem{mydefi}[mytheo]{Definition}

\theoremstyle{remark}

\numberwithin{equation}{section}
\numberwithin{figure}{section}

\begin{document}

\author{Matthew P. Young}
 \address{Department of Mathematics \\
 	  Texas A\&M University \\
 	  College Station \\
 	  TX 77843-3368 \\
 		U.S.A.}

 \email{myoung@math.tamu.edu}
 \thanks{This material is based upon work supported by the National Science Foundation under agreement No. DMS-2001306.  Any opinions, findings and conclusions or recommendations expressed in this material are those of the authors and do not necessarily reflect the views of the National Science Foundation.  }

\begin{abstract} 
We improve on the spectral large sieve inequality for symmetric-squares.  We also prove a lower bound showing that the most optimistic upper bound is not true for this family.
\end{abstract}

 \title{On the spectral large sieve inequality for symmetric-squares}
\maketitle
\section{Introduction}
\subsection{Notation and motivation}
Suppose that $\mathcal{F}$ is a family of automorphic forms with associated $L$-functions $L(f,s) = \sum_{n} \lambda_f(n) n^{-s}$, $f \in \mathcal{F}$.  We normalize $L(f,s)$ to have a functional equation under $s \rightarrow 1-s$.  We define a norm on a bilinear form associated to $\mathcal{F}$ by
\begin{equation}
\label{eq:bilineardef}
 \mathcal{B}(\mathcal{F}, N) = 
 \max_{|{\bf a}| = 1} 
 \sum_{f \in \mathcal{F}}
 \Big| \sum_{N \leq n \leq 2N} a_n \lambda_f(n)\Big|^2,
\end{equation}
where ${\bf a} = (a_n)$ and $|{\bf a}|^2 = \sum_{N \leq n \leq 2N} |a_n|^2$.  In practice, $\mathcal{F}$ will be a finite part of a larger infinite family, which we finitize by specifying the size of the conductors.  
It is also common to study modified versions of \eqref{eq:bilineardef}, which could include an arithmetical weight such as $w_{f}^{-1}$, where $w_f = \text{Res}_{s=1} L(f \otimes \overline{f}, s)$.

A large sieve inequality for $\mathcal{F}$ is an upper bound on $\mathcal{B}(\mathcal{F}, N)$.  
A strong bound reflects orthgonality properties of the coefficients $\lambda_f(n)$ and has applications to moments of $L$-functions, the distribution of their zeros, etc. 
We also seek \emph{lower} bounds on $\mathcal{B}(\mathcal{F}, N)$, which may be helpful for discerning the true size of $\mathcal{B}(\mathcal{F}, N)$.  By general principles of bilinear forms (see \cite[Chapter 7.3]{IK}), it is well-known that $\mathcal{B}(\mathcal{F}, N) \gg (|\mathcal{F}| + N)^{1-\varepsilon}$.
\begin{mydefi}
 The \emph{optimistic bound} (for $\mathcal{F}$) is the purported inequality
 \begin{equation}
  \mathcal{B}(\mathcal{F}, N) \ll (|\mathcal{F}| + N)^{1+\varepsilon},
 \end{equation}
 valid for all $N$.
\end{mydefi}
In some important cases, notably for families of Dirichlet characters, the optimistic bound is known to be true (e.g., see Theorem \ref{thm:Gallagher} and \eqref{eq:quadraticlargesieve}  below).  For some other families, the optimistic bound does not hold.  See Section \ref{section:lowerbound} below for more discussion on known counterexamples.

It appears to be difficult to even conjecture the true size of $\mathcal{B}(\mathcal{F}, N)$ for general families.

The optimistic bound is true for the following  
$GL_2$ family.  Let $u_j$ run over the Hecke-Maass cusp forms on $SL_2(\mathbb{Z})$, of Laplace eigenvalue $1/4 + t_j^2$.  Let $\lambda_j(n)$ denote the $n$-th Hecke eigenvalue of $u_j$, and set $w_j = \text{Res}_{s=1} L(s, u_j \otimes u_j)$.  We let $E(z, 1/2+it)$ denote the Eisenstein series having Hecke eigenvalues $\tau_{it}(n) = \sum_{ab=n} (a/b)^{it}$. 
The spectral large sieve inequality (see \cite{IwaniecSpectralLargeSieve, Jutila}) states
\begin{equation}
\label{eq:largesieveMaass}
\max_{|{\bf a}| = 1}
\sum_{T \leq t_j \leq T + \Delta} w_j^{-1}
\Big|
\sum_{N \leq n \leq 2N}
a_n \lambda_j(n) 
\Big|^2
\ll (\Delta T + N)(NT)^{\varepsilon},
\end{equation}
for $1 \leq \Delta \leq T$.  One may additionally incorporate a contribution from the Eisenstein series, without altering the bound on the right hand side of \eqref{eq:largesieveMaass}.

The present state of knowledge of large sieve inequalities for $GL_3$ families is much weaker than for $GL_1$ and $GL_2$.  Some references in this direction include \cite{DukeKowalski, Blomer, YoungGL3Kuz, BlomerButtcane, YoungGL3Spectral}.
The symmetric-squares of $SL_2(\mz)$ Maass forms are known to correspond to self-dual $SL_3(\mz)$ Maass forms \cite{GelbartJacquet, Soudry}, and form a particularly interesting $GL_3$ family.  The relation between the Dirichlet series coefficients of $L(u_j, s)$ 
and $L(\mathrm{sym}^2 u_j, s)$
is encapsulated in the identity
\begin{equation}
\label{eq:sym2DirichletSeriesDef}
\lambda_{\mathrm{sym}^2 u_j}(n) = \sum_{d^2 m = n} \lambda_j(m^2)
\end{equation}

The main objective of this paper is a bound on $\mathcal{M}_0(\Delta, T, N)$ defined by
\begin{equation}
\mathcal{M}_0(\Delta, T,N) = 
\max_{|{\bf a}| = 1}
\sum_{T \leq t_j \leq T + \Delta} w_j^{-1}
\Big|
\sum_{N \leq n \leq 2N}
a_n \lambda_j(n^2) 
\Big|^2.
\end{equation}
We also define the analogous contribution from the Eisenstein series, namely
\begin{equation}
\label{eq:Mcdef}
\mathcal{M}_{\infty}(\Delta, T, N) = 
\max_{|{\bf a}| = 1}
\int_{T \leq t \leq T + \Delta} w_t^{-1}
\Big|
\sum_{N \leq n \leq 2N}
a_n \tau_{it}(n^2) 
\Big|^2,
\end{equation}
where $w_t = |\zeta(1+2it)|^2$.  Let $\mathcal{M}(\Delta, T, N) = \mathcal{M}_0(\Delta, T, N) + \mathcal{M}_{\infty}(\Delta, T, N)$.
It follows from \eqref{eq:largesieveMaass} (after inclusion of the Eisenstein series)  that
\begin{equation}
\label{eq:Sym2TrivialBound}
\mathcal{M}(\Delta, T, N) \ll (\Delta T + N^2)(NT)^{\varepsilon},
\end{equation}
which matches the optimistic bound only for small values of $N$.  We regard \eqref{eq:Sym2TrivialBound} as a benchmark which does not use the arithmetical structure of the inner sum, namely that the Hecke eigenvalues are sampled at squares.

In a slightly different context, Duke and Kowalski \cite{DukeKowalski} showed a bound that is strongest for large values of $N$.  Their proof proceeds by duality and uses analytic properties of the degree $9$ Rankin-Selberg $L$-function $L(\mathrm{sym}^2 u_j \otimes \mathrm{sym}^2 u_k, s)$.  More precisely, they use a contour-shifting argument and apply the convexity bound for these Rankin-Selberg $L$-functions.  Strictly speaking, they study the level aspect of the problem and not the archimedean aspect which is the focus here.  Nevertheless, their method leads to
\begin{equation}
\label{eq:DukeKowalskiBound}
\mathcal{M}_{0}(\Delta, T, N) \ll (N + \Delta^{3/2} T^{5/2} N^{1/2}) (NT)^{\varepsilon}.
\end{equation}
This sets another useful frame of reference to complement \eqref{eq:Sym2TrivialBound}.
Junehyuk Jung and Min Lee (personal communication) have recently improved on \eqref{eq:DukeKowalskiBound}, using the dual approach.

We also mention that \cite[Problem 7.29]{IK} asks for an improvement on \eqref{eq:DukeKowalskiBound} (though technically 
they state the problem for the level aspect).

\subsection{Statements of results}
The main result of this paper is an improvement on \eqref{eq:Sym2TrivialBound} and \eqref{eq:DukeKowalskiBound} in many ranges.
\begin{mytheo}
\label{thm:mainthm}
Let notation be as above.  Then
\begin{equation}
\label{eq:mainthm}
\mathcal{M}(\Delta, T, N)
\ll (NT)^{\varepsilon}
\begin{cases}
\Delta T + T^{1/2} N, \qquad &N \leq T, \\
\Delta N + N^{3/2}, \qquad &T \leq N \leq T^2, \\
\frac{N^2}{T}, \qquad & T^2 \leq N.
\end{cases}
\end{equation}
\end{mytheo}
One may check that Theorem \ref{thm:mainthm} improves on (or agrees with) \eqref{eq:Sym2TrivialBound} for all $N$ and $T$.  Likewise, one may check that Theorem \ref{thm:mainthm} improves on \eqref{eq:DukeKowalskiBound} for $N \ll \Delta T^{7/3}$.  Another way to gauge Theorem \ref{thm:mainthm} is to ask in what ranges of parameters does it give the optimistic bound;  this occurs for $N \ll \min(T, \Delta T^{1/2})$, while \eqref{eq:Sym2TrivialBound} only matches the optimistic bound for $N \ll \Delta^{1/2} T^{1/2}$.
On the other hand, Theorem \ref{thm:lowerboundEisPart} below gives a \emph{lower} bound on $\mathcal{M}(\Delta, T, N)$ of size $\Delta N$, and Theorem \ref{thm:mainthm} appears even more favorable in this light.

As an aside, \eqref{eq:mainthm} is concisely expressed by
\begin{equation}
\label{eq:mainthmAlt}
\mathcal{M}(\Delta, T, N) 
\ll (NT)^{\varepsilon}
\Big( \Delta(T+N) + N (T+N)^{1/2} + \frac{N^2}{T} \Big).
\end{equation}

As a simple corollary, we record 
\begin{mycoro}
\label{coro:secondmoment}
For $T^{1/2} \leq \Delta \leq T$ we have
\begin{equation}
\sum_{T \leq t_j \leq T+\Delta} |L(\mathrm{sym}^2 u_j, 1/2)|^2 \ll \Delta T^{1+\varepsilon}.
\end{equation}
\end{mycoro}
Corollary \ref{coro:secondmoment} is not close to the state of the art, as Lam \cite{Lam} previously obtained this quality of bound with $\Delta = T^{1/3}$, which  was recently improved further to $\Delta = T^{1/5}$ in \cite{KhanYoung}.  On the other hand, \eqref{eq:Sym2TrivialBound} gives 
a Lindel\"{o}f-on-average bound only for $\Delta = T$.  In addition, it is likely that some of the intermediate steps used in the proof of Theorem \ref{thm:mainthm} could be useful for some applications.  For instance, \eqref{eq:Kformula1} or \eqref{eq:Kformula2}, which are formulas valid for arbitrary coefficients $a_n$, could be followed up with tools that use special properties of the coefficients, such as a Poisson/Voronoi summation formula.
This strategy was used in \cite{YoungGL3GL2specialPoints}.

Theorem \ref{thm:mainthm} treats the Hecke eigenvalues sampled at squares, which is not the same as the Dirichlet series coefficients of the symmetric-square $L$-function (recall \eqref{eq:sym2DirichletSeriesDef}).  We next discuss the connections between these objects.
Define
\begin{equation}
 \mathcal{M}_0^{(2)}(\Delta, T,N) = 
\max_{|{\bf a}| = 1}
\sum_{T \leq t_j \leq T + \Delta} w_j^{-1}
\Big|
\sum_{N \leq n \leq 2N}
a_n \lambda_{\mathrm{sym}^2 u_j}(n) 
\Big|^2.
\end{equation}
Suppose for convenience that $a_n$ is supported on $[N, 2N]$.  Then note
\begin{equation}
\sum_{ n  } a_n \lambda_{\mathrm{sym}^2 u_j}(n)
= \sum_{d} b_d \lambda_j(d^2),
\qquad
\text{where}
\qquad
b_d = \sum_{k} a_{dk^2}.
\end{equation}
Note that if $a_n$ is restricted to square-free integers, then $\lambda_{\mathrm{sym}^2 u_j}(n) = \lambda_j(n^2)$, so a bound on $\mathcal{M}_{0}(\Delta, T, N)$ may be applied under this assumption.
Without this square-free restriction, we may deduce the following.
\begin{mycoro}
\label{coro:mainthmSym2}
 For any ${\bf a}$, we have
 \begin{equation}
 \label{eq:mainthmSym2}
  \sum_{T \leq t_j \leq T + \Delta} w_j^{-1}
\Big|
\sum_{N \leq n \leq 2N}
a_n \lambda_{\mathrm{sym}^2 u_j}(n) 
\Big|^2
\ll
\mathcal{D} 
+
\Big(\Delta N + N (T+N)^{1/2} + \frac{N^2}{T}\Big) (NT)^{\varepsilon}
|{\bf a}|^2,
 \end{equation}
 where
 \begin{equation}
 \mathcal{D} = \Delta T \sum_{n} \Big| \sum_k a_{n k^2}\Big|^2
 \ll \Delta T N^{1/2} |{\bf a}|^2.
 \end{equation}
 In particular,
  \begin{equation}
 \label{eq:M0sym2bound}
  \mathcal{M}_0^{(2)}(\Delta, T, N) 
\ll (NT)^{\varepsilon}
\Big( \Delta T N^{1/2} + \Delta N + N (T+N)^{1/2} + \frac{N^2}{T} \Big).
 \end{equation}
\end{mycoro}

Thorner and Zaman \cite{ThornerZaman} have proved a complementary bound on $\mathcal{M}_0^{(2)}$ by the dual approach.  By following their method, they implicitly show 
\begin{equation}
\mathcal{M}_0^{(2)}(\Delta, T, N) 
\ll (NT)^{\varepsilon} (N + \Delta^2 T^4).
\end{equation}



The shape of $\mathcal{D}$ in \eqref{eq:mainthmSym2} is reminiscent of the family of quadratic Dirichlet characters, as we now elaborate.  Heath-Brown \cite{HB} showed
\begin{equation}
\label{eq:quadraticlargesieve}
\sumstar_{d \leq X}
\Big| \sumstar_{n \leq N} a_n \chi_d(n) \Big|^2
\ll (X+N)^{1+\varepsilon} |{\bf a}|^2,
\end{equation}
where the sums restrict $d,n$ to odd, square-free integers.  The square-free restriction is vital:
if $d$ is allowed to run over squares, then one can produce a large term of size $\sqrt{X} N |{\bf a}|^2$ by taking $a_n = 1$ for all $n$.  A term of this size would contradict \eqref{eq:quadraticlargesieve}.  By the duality principle, one may produce a large term of size $\sqrt{N} X |{\bf a}|^2$ if $n$ is not restricted to square-free integers (see also discussion surrounding \eqref{eq:BiasedBound} below).  As a heuristic, one might expect
\begin{equation*}
\frac{1}{X^*} \sumstar_{d \leq X} \chi_d(m n) \approx  \delta_{mn=\square}, 
\qquad 
\text{where}
\qquad
X^* = \sumstar_{d \leq X} 1,
\end{equation*}
and consequently,
\begin{equation*}
\sumstar_{d \leq X} \Big|\sum_{n \leq N} a_n \chi_d(n) \Big|^2
\approx X^* \sum_n \Big| \sum_k a_{k^2 n} \Big|^2.
\end{equation*}
The point is that this shape of the ``diagonal" term matches that in Corollary \ref{coro:mainthmSym2}.  See Proposition \ref{prop:sym2lowerbound} below for more discussion.

\subsection{Lower bounds}
It is unclear what to conjecture for the true size of $\mathcal{M}(\Delta, T, N)$.  In other contexts, notably the family of cusp forms on $\Gamma_1(q)$ studied by Iwaniec and Li \cite{IwaniecLi}, the optimistic bound is not true.  Blomer and Buttcane \cite{BlomerButtcane} showed that for spectral families on $SL_n(\mathbb{Z})$, with $n \geq 3$, the Eisenstein series component contributes a large term, implying that the optimistic bound is not true for these families.  The recent work \cite{YoungGL3Spectral} showed that, at least for $n=3$, the Blomer-Buttcane bound can be improved by restricting to the cusp forms (that is, omitting the Eisenstein series).  This raises the question on what are the sizes of $\mathcal{M}_0(\Delta, T, N)$ and $\mathcal{M}_{\infty}(\Delta, T, N)$ separately. 
Towards this end, we have
\begin{mytheo}
\label{thm:lowerboundEisPart}
If $\Delta \gg (NT)^{\delta}$ for some $\delta >0$, then
\begin{equation}
\mathcal{M}_{\infty}(\Delta, T, N) \gg \Delta N^{} (NT)^{-\varepsilon}.
\end{equation}
\end{mytheo}
Since the method of proof of Theorem \ref{thm:mainthm} treats $\mathcal{M}_0 + \mathcal{M}_{\infty}$, Theorem \ref{thm:lowerboundEisPart} shows that we cannot remove the term $\Delta N$ in \eqref{eq:mainthmAlt} (also seen heuristically in \eqref{eq:sketchpartchiTrivialCNterm} below).  However, since this lower bound comes from $\mathcal{M}_{\infty}$, this has no direct implication about $\mathcal{M}_{0}$ itself.
A short proof of Theorem \ref{thm:lowerboundEisPart} appears in Section \ref{section:lowerbound}.

Next we discuss the true size of $\mathcal{M}_0^{(2)}$.
On first inspection, \eqref{eq:M0sym2bound} appears to be defective in that the term $\Delta T N^{1/2}$ is much larger than the size of the family.  However, we have the following lower bound showing this term cannot be improved:
\begin{myprop}
\label{prop:sym2lowerbound}
 If $N \ll \Delta T^{1-\delta}$ for some $\delta > 0$, then
 \begin{equation}
   \mathcal{M}_0^{(2)}(\Delta, T, N) \gg  \Delta T N^{1/2} (NT)^{-\varepsilon}.
 \end{equation}
\end{myprop}
Note that this lower bound does \emph{not} come from Eisenstein series, so this is a different phenomenon than that observed by Blomer and Buttcane \cite{BlomerButtcane}.  Nevertheless, both Theorems \ref{thm:lowerboundEisPart} and Proposition \ref{prop:sym2lowerbound} have similar proofs, as we next discuss.

Let $\mathcal{N} \subseteq [N,2N] \cap \mathbb{Z}$.  We say that the set $\mathcal{N}$ is \emph{biased} for the family $\mathcal{F}$ if
\begin{equation}
\sum_{f \in \mathcal{F}} \Big| \sum_{n \in \mathcal{N}} \lambda_f(n) \Big|^2 \gg
  (|\mathcal{F}| \cdot |\mathcal{N}|^2)^{1-\varepsilon}.
\end{equation}
By taking ${\bf a}$ to be the characteristic function of $\mathcal{N}$, we see that if $\mathcal{F}$ has a biased set $\mathcal{N}$, then
\begin{equation}
\label{eq:BiasedBound}
\mathcal{B}(\mathcal{F}, N) \gg (|\mathcal{F}| \cdot |\mathcal{N}|)^{1-\varepsilon},
\end{equation}
which contradicts the optimistic bound for $\mathcal{F}$ provided $|\mathcal{N}| \gg (|\mathcal{F}| \cdot N)^{\delta}$ for some $\delta>0$.  As a simple example of a biased set, let $\mathcal{F}$ consist of the quadratic Dirichlet characters of odd conductor $d$ with $d \asymp X$, and let $\mathcal{N}$ consist of the squares in  $[N, 2N]$, so $|\mathcal{N}| \asymp \sqrt{N}$.  It is easy to see that 
$\mathcal{N}$ is biased for this family.  Of course, this does not contradict \eqref{eq:quadraticlargesieve} since there $n$ is restricted to square-free integers.

Next we explain the source of bias in the two families in Theorems \ref{thm:lowerboundEisPart} and Proposition \ref{prop:sym2lowerbound}.  For the family in Theorem \ref{thm:lowerboundEisPart}, let $\mathcal{N}$ consist of primes in  $[N, 2N]$, so $|\mathcal{N}| \asymp \frac{N}{\log{N}}$.  The bias arises from $\tau_{it}(p^2) = 1 + p^{2it} + p^{-2it}$, which is approximately $1$ on average over $t \in [T, T+\Delta]$.
Similarly, for the symmetric-square family in Proposition \ref{prop:sym2lowerbound}, we let $\mathcal{N}$ consist of $p^2$ with $p$ prime and $p^2 \in [N, 2N]$, so  $|\mathcal{N}| \asymp \frac{N^{1/2}}{\log{N}}$.  The bias here may be seen from $\lambda_{\mathrm{sym}^2 u_j}(p^2) = 1+\lambda_j(p^2)$, which again is approximately $1$ on average over $t_j \in [T, T+\Delta]$.  
More details of the proof are presented in Section \ref{section:lowerbound}.

This is some evidence that the definition \eqref{eq:bilineardef} might require some modification, depending on some arithmetical features of the family $\mathcal{F}$.  More examples of families with biased sets would be welcome.

We also mention that Dunn and Radziwi\l \l {} \cite{DunnRadziwill} have shown that the optimstic bound does not hold for the family of cubic residue symbols.

\subsection{Sketch of proof of Theorem \ref{thm:mainthm}}
\label{subsection:sketch}
The Kuznetsov formula  leads to a diagonal term of size $\Delta T$ as well as a sum of Kloosterman sums roughly of the shape
\begin{equation}
\label{eq:sketchpartAfterKuz}
\frac{\Delta T}{N \sqrt{C}} \sum_{c \sim C} \sum_{m,n} a_m \overline{a_n} S(m^2,n^2;c) e_c(2mn) e^{i\frac{T^2 c}{mn}},
\end{equation}
where $C$ runs over dyadic segments with $1 \ll C \ll \frac{N^2}{\Delta T}$.
Here $e_c(2mn)$ comes from the main part of the phase of the Bessel-type transform from the Kuznetsov formula, and the factor $\exp(i\frac{T^2 c}{mn})$ is a lower-order term in the phase.  Since $e_c(2mn)$ is periodic in $m,n$ modulo $c$, it can be joined with the Kloosterman sum $S(m^2, n^2;c)$.  If $(mn,c) = 1$ then with $w=mn$ we can expand $F(w) = S(w^2, 1;c) e_c(2w)$ into Dirichlet characters, via
\begin{equation}
\label{eq:sketchpartFourier}
F(w) = \sum_{\chi \shortmod{c}} \widehat{F}(\chi) \chi(w),
\quad
\text{where}
\quad
\widehat{F}(\chi) = \frac{1}{\varphi(c)} \sum_{u \shortmod{c}} \overline{\chi}(u) S(u^2,1;c) e_c(2u).
\end{equation}
We similarly have an archimedean separation of variables by the Mellin transform, via
\begin{equation}
e^{i\frac{T^2 c}{mn}} \approx \frac{1}{\sqrt{P}} \int_{t \asymp P} f(t) \Big(\frac{T^2 c}{mn}\Big)^{-it} dt, \qquad P = 1+ \frac{T^2 C}{N^2},
\end{equation}
where $f$ is a smooth function satisfying $f(t) \ll 1$.  
Inserting these into \eqref{eq:sketchpartAfterKuz}, we obtain
\begin{equation}
\label{eq:sketchpartAfterFourier}
\frac{\Delta T}{N \sqrt{CP}} \int_{t \asymp P} f(t) T^{-2it}  \sum_{c \sim C}  c^{-it}
\sum_{\chi \shortmod{c}} \widehat{F}(\chi) 
\sum_{m,n} a_m \overline{a_n} \chi(mn) (mn)^{it} dt.
\end{equation}

As a heuristic, consider the contribution to \eqref{eq:sketchpartAfterFourier} from $\chi$ primitive modulo $c$.  For such $\chi$, $|\widehat{F}(\chi)| \ll c^{\varepsilon}$.  Then the classical hybrid large sieve (see Theorem \ref{thm:Gallagher}) gives a bound
\begin{equation}
\label{eq:sketchpartBoundchiPrimitive}
\frac{\Delta T}{N \sqrt{CP}} (C^2 P + N) |{\bf a}|^2.
\end{equation}
Note the latter term is $\frac{\Delta T}{\sqrt{CP}} \ll \Delta T$, which is already bounded by the diagonal term.  The former term is bounded by $\frac{\Delta T}{N} \frac{T}{N} \frac{N^4}{\Delta^2 T^2} = \frac{N^2}{\Delta}$.

On the opposite extreme from $\chi$ primitive is $\chi$ trivial.  It turns out (see Lemma \ref{lemma:FhatchiTrivialPrimeModulusBound}) that the most significant contribution from trivial $\chi$ comes from $c$ ranging over squares, in which case $\widehat{F}(\chi_0) = c^{1/2}$.  These terms contribute to \eqref{eq:sketchpartAfterFourier} an expression bounded by
\begin{equation}
\label{eq:sketchpartchiTrivial}
\frac{\Delta T}{N \sqrt{CP}} \int_{t \asymp P}   \sum_{\substack{c \sim C \\ c = \square}}  \sqrt{c}
|\sum_{m} a_m m^{it} \Big|^2 dt.
\end{equation}
The hybrid large sieve inequality in this case reduces to the mean value theorem for Dirichlet polynomials (see \cite[Theorem 9.1]{IK}), and produces a bound of the form
\begin{equation}
\label{eq:sketchpartchiTrivialSimplified}
\frac{\Delta T}{N \sqrt{CP}} (CP + CN) |{\bf a}|^2.
\end{equation}
Comparing this with \eqref{eq:sketchpartBoundchiPrimitive}, we see that the term $CP$ in \eqref{eq:sketchpartchiTrivialSimplified} is dwarfed by the term $C^2 P$ in \eqref{eq:sketchpartBoundchiPrimitive}.  However, the term $CN$ is now larger than the term $N$.  This term with $CN$ contributes
\begin{equation}
\label{eq:sketchpartchiTrivialCNterm}
\frac{\Delta T}{N \sqrt{CP}} CN = \Delta N.
\end{equation}

Now taking into account the bounds from the diagonal, the primitive characters, and the trivial characters, we obtain a bound of size
\begin{equation}
\label{eq:sketchpartPenultimateBound}
\Delta T + \Delta N + \frac{N^2}{\Delta}.
\end{equation}
Note that \eqref{eq:sketchpartPenultimateBound} has the feature that one part is increasing in $\Delta$, and another part is decreasing in $\Delta$.   This type of structure is commonly seen in harmonic analysis in association with the duality principle.  
It should thus not be surprising that \eqref{eq:sketchpartPenultimateBound} may be improvable by enlarging $\Delta$.  When $N \leq T$ then we replace $\Delta$ by $\Delta + \frac{N}{\sqrt{T}}$, which leads to the first bound in Theorem \ref{thm:mainthm}.  When $T \leq N \leq T^2$ then we replace $\Delta$ by $\Delta + \sqrt{N}$, leading to the second bound in Theorem \ref{thm:mainthm}.  Finally, when $N \geq T^2$ we replace $\Delta$ by $T$, giving the third bound in Theorem \ref{thm:mainthm}.

This discussion indicates that the rigorous bound from Theorem \ref{thm:mainthm} matches the bounds from considering the primitive characters and the trivial characters separately, which is encouraging.  Most of the work in this paper boils down to treating all the cases in a uniform manner, including dealing with $m,n$ that are not coprime to $c$, and characters $\chi$ that are neither primitive nor trivial.

\subsection{Remarks on possible improvements}
Since the term of size $\Delta N$ in Theorem \ref{thm:mainthm} matches in rough order of magnitude the lower bound on $\mathcal{M}_{\infty}$, it is natural to ask if the term $\Delta N$ is required in bounding $\mathcal{M}_0$, or if an improvement is possible.
A possible method to improve on the bound on $\mathcal{M}_0$  would be to continue to use the Kuznetsov formula, but to cancel (perhaps only partially) the contribution of the Eisenstein series from some part of the sum of Kloosterman sums.  Luo \cite{Luo} achieved this type of cancellation in a different spectral large sieve problem.  

As a possible clue in this direction, consider the contribution from $c=p^2$ with $p$ prime in \eqref{eq:sketchpartAfterKuz}.  The Kloosterman sum of modulus $p^2$ can be calculated in closed form (see \cite[Section 4.3]{IwaniecClassical}), showing $S(m^2, n^2;p^2) = p \sum_{\pm } e_p(\pm 2mn)$ for $(p, mn) = 1$.  The term with $e_p(-2mn)$ cancels the phase from $e_p(2mn)$ coming from the Bessel transform, and leads to an expression resembling
\begin{equation}
\frac{\Delta T}{N} \sum_{p} \sum_{m,n} a_m \overline{a_n} e^{i\frac{T^2 p^2}{mn}}.
\end{equation}
When $p \ll \frac{N}{T}$, the exponential is not oscillatory, so one sees a sum roughly of the form $\Delta |\sum_n a_n|^2$.  This now closely resembles the term \eqref{eq:remarksIntroCtsMainTerm?} responsible for the lower bound in Theorem \ref{thm:lowerboundEisPart}.

Another possible inroad on this problem would be to consider the holomorphic weight $k$ cusp form analog of the problem.  With only minor modifications of the proof, one could derive the analog of Theorem \ref{thm:mainthm} for this family. A gain is that there is no Eisenstein contribution, so there would be no need to cancel a contribution from the Eisenstein spectrum.  This \emph{might} be a hint that the term of size $\Delta N$ is not removable for the cusp forms.

\section{Preliminaries}
In this section we collect some tools needed for the proof of Theorem \ref{thm:mainthm}.
\subsection{Around the Kuznetsov formula}
\begin{myprop}[Kunzetsov formula]
Suppose that $h(t)$ is even, holomorphic in $|\text{Im}(t)| \leq \frac12 + \delta$, and satisfying $h(t) \ll (1+|t|)^{-2-\delta}$, for some $\delta > 0$.  Then for certain weights $\omega_j$, $\omega_t$ proportional to $w_j$ and $w_t$, respectively, we have
\begin{multline}
\sum_{j} \omega_j^{-1} \lambda_j(m) \lambda_j(n)  h(t_j) 
+ 
\int_{-\infty}^{\infty} \omega_t^{-1} \tau_{it}(m) \tau_{it}(n) h(t) dt 
\\
= \delta_{m=n} \frac{1}{\pi} \intR h(r) r \tanh(\pi r) dr
+ \sum_{c=1}^{\infty} c^{-1} S(m,n;c) H_{\infty}\Big(\frac{4 \pi \sqrt{mn}}{c}\Big),
\end{multline}
where
\begin{equation}
H_{\infty}(x)
= 2i \intR \frac{J_{2ir}(x)}{\cosh(\pi r)} r h(r) dr.
\end{equation}
\end{myprop}

\begin{mylemma}[Properties of $H_{\infty}$]
\label{lemma:HinfinityProperties}
Let $T^{\delta} \leq \Delta \leq T^{1-\delta}$ for some $\delta>0$, and let 
$$h(t) = \frac{t^2+\frac14}{T^2} \Big[\exp(-(t-T)^2/\Delta^2) + \exp(-(t+T)^2/\Delta^2) \Big].$$  
Then $H_{\infty}(x)$ is very small unless
\begin{equation}
\label{eq:HinfinityTruncationSpot}
x \gg \Delta T^{1-\varepsilon}.
\end{equation}
Moreover, for $x$ satisfying \eqref{eq:HinfinityTruncationSpot}, $H_{\infty}(x)$ has an asymptotic formula of the form
\begin{equation}
H_{\infty}(x) = \sum_{\pm} \frac{\Delta T}{\sqrt{x}} e^{\pm ix}  e^{i \phi(x)} I(x) + O(T^{-A}),
\end{equation}
where $I = I_{\pm, A}$ is a smooth function satisfying $I^{(j)}(x) \ll_{j,A} x^{-j}$ for $x \gg \Delta T^{1-\varepsilon}$, and $\phi = \phi_{\pm}$ is a smooth function with an asymptotic expansion of the form $\phi(x) = c_1 \frac{T^2}{x} + c_2 \frac{T^4}{x^3} + \dots$.
\end{mylemma}
This lemma can be extracted from the work of Jutila and Motohashi \cite[pp.75--76]{JM}.

\subsection{The $GL_1$ large sieve}
\begin{mytheo}[Gallagher \cite{Gallagher}]
\label{thm:Gallagher}
Let $Q, T \geq 1$.  Then for any vector ${\bf a} = (a_n)$, we have
\begin{equation}
\int_{0}^{T} \sum_{q \leq Q}  \medspace\sumstar_{\chi \shortmod{q}} 
\Big| \sum_{n \leq N} a_n \chi(n) n^{it} \Big|^2 dt
\ll (Q^2 T + N) |{\bf a}|^2.
\end{equation}
\end{mytheo}
We also need the following variant.
\begin{mycoro}
\label{coro:LargeSieveVariant}
Let $Q, T, d, g \geq 1$, with $g$ square-free.  
Define $d'$ to be the smallest integer so that $d|(d')^2$ (so
if $d = \prod_{p|d} p^{d_p}$, then 
 $d' = \prod_p p^{\lceil d_p/2 \rceil}$).
Then for any vector ${\bf a} = (a_n)$, we have
\begin{equation}
\label{eq:LargeSieveVariant}
\int_{0}^{T} \sum_{q \leq Q}  \medspace\sumstar_{\chi \shortmod{q}} 
\Big| \sum_{\substack{m,n \leq N \\ d|(m^2, n^2) \\ g| mn/d}} a_m \overline{a_n} \chi(mn) (mn)^{it} \Big| dt
\ll N^{\varepsilon} 
 \Big(Q^2 T + \frac{N}{d'}\Big) \sum_n |a_{nd'}|^2.
\end{equation}
\end{mycoro}
\begin{proof}
First set $h= (m,n)$, and change variables $m \rightarrow hm$ and $n \rightarrow hn$.  
After this change, the condition $d|(m^2, n^2)$ becomes $d|h^2$.  Likewise, the condition $g|\frac{mn}{d}$ becomes $mn \equiv 0 \pmod{\frac{g}{(g,h^2/d)}}$. 
Since $(m,n) = 1$, we can parameterize the solutions to this latter congruence by writing $g_1 g_2 = \frac{g}{(g, h^2/d)}$ and imposing $m \equiv 0 \pmod{g_1}$ and $n \equiv 0 \pmod{g_2}$.
Then we have
\begin{equation}
\Big| \sum_{\substack{m,n \leq N \\ d|(m^2, n^2) \\ g|mn/d}} a_m \overline{a_n} \chi(mn) (mn)^{it} \Big|
\leq
\sum_{d|h^2}
\sum_{g_1 g_2 = \frac{g}{(g,h^2/d)}} 
\Big| \sum_{\substack{(mg_1,ng_2) = 1 
}} a_{mg_1 h} \overline{a_{n g_2 h}} \chi(mn) (mn)^{it} \Big|.
\end{equation}
We then apply M\"obius inversion to detect the condition $(m,n) = 1$ and thereby separate the variables.  After this step, we then apply Theorem \ref{thm:Gallagher}.  This gives that the left hand side of \eqref{eq:LargeSieveVariant} is bounded by
\begin{equation}
\sum_{d|h^2}
\sum_{g_1 g_2 = \frac{g}{(g,h^2/d)}} 
\sum_{\ell} |\mu(\ell)| \Big(Q^2 T + \frac{N}{h g_1 \ell} \Big) \sum_n |a_{n h g_1 \ell}|^2.
\end{equation}
The condition $h^2 \equiv 0 \pmod{d}$ is equivalent to $h \equiv 0 \pmod{d'}$.  Simplifying the above expression using a divisor function bound leads quickly to \eqref{eq:LargeSieveVariant}.
\end{proof}

\section{Structural steps}
\subsection{Reduction}
We now embark on the proof of Theorem \ref{thm:mainthm}.  Our first stage 
parallels the Fourier/Mellin decomposition indicated in Section \ref{subsection:sketch}.
  We insert the weight function $h(t)$ from Lemma \ref{lemma:HinfinityProperties}, extend the spectral sum/integral to all $t_j$, $t$, open the square, and apply the Kuznetsov formula.  The diagonal term is of size $\Delta T$.  We then obtain $\mathcal{M}(\Delta, T, N) \ll \Delta T + \max_{|{\bf a}| = 1} |\mathcal{K}(\Delta, T, N)|$, where
\begin{equation}
\mathcal{K}(\Delta, T, N) = 
\sum_{m,n} a_m \overline{a_n} 
\sum_{c=1}^{\infty} \frac{S(m^2, n^2;c)}{c} H_{\infty}\Big(\frac{4 \pi mn}{c}\Big).
\end{equation}
According to Lemma \ref{lemma:HinfinityProperties}, we write $H_{\infty}(x) = e^{ix} H_{+}(x) + e^{-ix} H_{-}(x) + O(T^{-A})$, and correspondingly write
$\mathcal{K} = \sum_{\pm} \mathcal{K}_{\pm} + O(T^{-A} |{\bf a}|^2)$.  We have
\begin{equation}
\label{eq:Kformula}
\mathcal{K}_{\pm}(\Delta, T, N)
=
\sum_{\pm}
 \sum_{m,n} a_m \overline{a_n} D_{\pm}(m,n), 
\end{equation}
where
\begin{equation}
 D_{\pm}(m,n) = \sum_{c=1}^{\infty} \frac{S(m^2, n^2;c) e_c(\pm 2mn)}{c} H_{\pm}\Big(\frac{4 \pi mn}{c}\Big).
\end{equation}
Both terms $\mathcal{K}_{\pm}$ may be bounded in the same way, as they are essentially complex conjugates of each other.  We then proceed with developing $D_{+}$, and drop the subscript $+$ from the notation.  

\begin{mylemma}
We have
\begin{equation}
\label{eq:DmnInitialDecomposition}
D(m,n) = \sum_{d|(m^2, n^2)}
\sum_{g|\frac{mn}{d}} \frac{\mu(g)}{g}
\sum_{(c, mn/d) = 1}  \frac{S((\frac{mn}{dg})^2, 1;c) e_c(\frac{2mn}{dg})}{c}
H\Big(\frac{4\pi mn}{dgc}\Big).
\end{equation}
\end{mylemma}
\begin{proof}
We begin with the Selberg identity \cite[(4.10)]{IwaniecClassical}  
$$S(m^2, n^2;c) = \sum_{d|(m^2, n^2, c)} d S(m^2 n^2/d^2, 1;c/d),$$
which quickly leads to
\begin{equation}
D(m,n) = \sum_{d|(m^2, n^2)} \sum_{c=1}^{\infty} \frac{S((mn/d)^2, 1;c) e_c(2mn/d)}{c} H\Big(\frac{4 \pi mn/d}{c}\Big).
\end{equation}
As shorthand, let $w=mn/d$, and parameterize the sum over $c$ by the value of $g=\gcd(w,c)$.
Then 
\begin{equation}
\sum_{c=1}^{\infty} \frac{S(w^2, 1;c) e_c(2w)}{c} H\Big(\frac{4 \pi w}{c}\Big)
=
\sum_{g|w} g^{-1}
\sum_{(c, w/g) = 1}  
\frac{S(w^2, 1;gc) e_c(2w/g)}{c} H\Big(\frac{4 \pi w/g}{c}\Big).
\end{equation}
To simplify this, we claim that if $(c,g) \neq 1$ then $S(w^2, 1;gc) = 0$.  In turn this follows from the fact that $S(p^k, 1;p^c) = 0$ if $c \geq 2$ and $k \geq 1$, which can be proved elementarily.  With the condition $(g,c) = 1$ in effect, we then have $S(w^2, 1;gc) = S(w^2 \overline{c}^2,1;g) S(w^2 \overline{g}^2, 1;c) = \mu(g) S((w/g)^2, 1;c)$.  At this point, direct substitutions complete the proof.
\end{proof}

\subsection{Finite Fourier analysis}
Our next stage of development is multiplicative Fourier analysis, mirroring \eqref{eq:sketchpartFourier} from the sketch.  As in the sketch, define $F(w) = S(w^2, 1;c) e_c(2w)$, where $(w,c) = 1$.  Then \eqref{eq:sketchpartFourier} holds, and with a change of variables we have
\begin{equation}
\label{eq:Fhatdef2}
\widehat{F}(\chi) 
= 
\frac{1}{\varphi(c)} \sum_{u, t \shortmod{c}} \overline{\chi}(ut) e_c(t (u+1)^2).
\end{equation}
Applying this to \eqref{eq:DmnInitialDecomposition}, with $w=\frac{mn}{dg}$ we obtain
\begin{equation}
\label{eq:DmnAfterFourierbeforeMellin}
D(m,n) = 
\sum_{d|(m^2, n^2)}
\sum_{g|\frac{mn}{d}} \frac{\mu(g)}{g}
\sum_{c =1}^{\infty} 
\sum_{\chi \shortmod{c}} 
\frac{\widehat{F}(\chi) \chi(mn \overline{dg})}{c}
H\Big(\frac{4\pi mn}{dgc}\Big).
\end{equation}
Section \ref{section:Fhat} is devoted to understanding $\widehat{F}(\chi)$ in greater detail.  

\subsection{Archimedean Fourier/Mellin analysis}
The representation \eqref{eq:DmnAfterFourierbeforeMellin}  neatly separates the variables $m,n$ with the multiplicative character $\chi$.  However, they remain joined inside $H(x)$, so we will use Mellin inversion as the analogous tool at the archimedean place.  Suppose that $1 = \sum_X W(x/X)$ is a smooth dyadic partition of unity, where $X$ runs over numbers of the form $2^{j/2}$ with $j \in \mz$, and $W$ has support on $[1,2]$.  
Define
\begin{equation}
\label{eq:GMellinRepresentation}
G(x) = G_X(x) = W(x/X) H(4 \pi x),
\quad \text{and}
\quad
\widetilde{G}(-s) = \int_0^{\infty} G(x) x^{-s} \frac{dx}{x}.
\end{equation}
\begin{mylemma}
\label{lemma:GMellinproperties}
Suppose $X \gg \Delta T^{1-\varepsilon}$, and let
\begin{equation}
P = 1 + \frac{T^2}{X}.
\end{equation}
Then $\widetilde{G}(-s)$ is analytic on $\mathbb{C}$ and satisfies
\begin{equation}
|\widetilde{G}(-\sigma-it)| \ll_{\sigma, A, \varepsilon} \Delta T \frac{X^{-\sigma}}{\sqrt{XP}} \Big(1+ \frac{|t|}{P T^{\varepsilon}}\Big)^{-A}.
\end{equation}
\end{mylemma}
\begin{proof}
Since $G$ is smooth and compactly supported on the positive reals, its Mellin transform is entire.
By Lemma \ref{lemma:HinfinityProperties}, we have 
\begin{equation}
\widetilde{G}(-s) = \Delta T \int_0^{\infty} x^{-1/2-s} W(x/X) I(x) e^{ic_1 \frac{T^2}{x} + \dots}  \frac{dx}{x}.
\end{equation}
If $X \gg T^{2-\varepsilon}$, i.e., $P \ll T^{\varepsilon}$, then the bound follows by standard integration by parts.  For $X \ll T^{2-\varepsilon}$, i.e., $P \gg T^{\varepsilon}$, 
the bound results from robust stationary phase analysis (see \cite{KPY}, for instance). 
\end{proof}

Applying \eqref{eq:GMellinRepresentation} to \eqref{eq:DmnAfterFourierbeforeMellin}, and interchanging the orders of summation and integration (valid for $\sigma >2$, using only a trivial bound $|\widehat{F}(\chi)| \leq c$), we deduce the following, which is similar in spirit to \cite[(3.1)]{Jutila}.
\begin{mylemma}
\label{lemma:DmnFourierMellin}
For $\sigma>2$, we have
\begin{equation}
\label{eq:DmnFourierMellin}
D(m,n) = 
\sum_{X \text{ dyadic}}
\sum_{d|(m^2, n^2)}
\sum_{g|\frac{mn}{d}} \frac{\mu(g)}{g} 
 \int_{(\sigma)} \widetilde{G}(-s)
\Big(\frac{mn}{dg}\Big)^s
\sum_{c =1}^{\infty} 
\sum_{\chi \shortmod{c}} 
\frac{\widehat{F}(\chi) \chi(mn \overline{dg})}{c^{1+s}} \frac{ds}{2 \pi i}.
\end{equation}
\end{mylemma}

\section{Properties of $\widehat{F}$}
\label{section:Fhat}

\begin{mylemma}
\label{lemma:absFhatIsMultiplicative}
Suppose $c=c_1 c_2$ with $(c_1, c_2) = 1$, and $\chi = \chi_1 \chi_2$ with $\chi_j$ modulo $c_j$ for $j=1,2$. Then $\widehat{F}(\chi)$ satisfies the twisted-multiplicativity relation
\begin{equation}
\label{eq:absFhatIsMultiplicative}
\widehat{F}(\chi)
=
\overline{\chi_1}(c_2) \overline{\chi_2}(c_1)
\widehat{F}(\chi_1) \widehat{F}(\chi_2),
\end{equation}
and so $|\widehat{F}(\chi)|$ is multiplicative.
\end{mylemma}
\begin{proof}
Standard with Chinese remainder theorem.
\end{proof}

In light of Lemma \ref{lemma:absFhatIsMultiplicative}, 
it suffices to understand $\widehat{F}(\chi)$ when $c= p^k$.

\begin{mylemma}[Primitive case]
\label{lemma:FhatchiPrimitiveModulusBound}
Suppose that $\chi$ has conductor $p^k$, with $k \geq 1$.  If $p=2$ then $\widehat{F}(\chi) = 0$.  If $p$ is odd and $\chi$ is not the Legendre symbol, then
\begin{equation}
|\widehat{F}(\chi)| = \frac{p}{p-1}.
\end{equation}
If $\chi$ is the Legendre symbol, then
\begin{equation}
\label{eq:FhathiLegendre}
|\widehat{F}(\chi)| = \frac{\sqrt{p}}{p-1}.
\end{equation}
\end{mylemma}
\begin{proof}
We use the formula \eqref{eq:Fhatdef2} and evaluate the $t$-sum in terms of a Gauss sum, giving
\begin{equation}
\widehat{F}(\chi) = \frac{\tau(\overline{\chi})}{\varphi(p^k)} 
\sum_{u \shortmod{p^k}} \overline{\chi}(u) \chi^2(u + 1)
= 
\chi(-1)
\frac{\tau(\overline{\chi}) J(\overline{\chi}, \chi^2) }{\varphi(p^k)} 
,
\end{equation}
where $J(\chi, \psi)$ is the Jacobi sum.
Note that if $p=2$ then $\widehat{F}(\chi) = 0$ since $2|u(u+1)$ for all $u$.  For $p$ odd then either $\chi^2$ has conductor $p^k$, or $\chi$ is the Legendre symbol (and so $k=1$).  In case $\chi^2$ has conductor $p^k$, then according to
\cite[(3.18)]{IK}, then $J(\overline{\chi}, \chi^2) = \frac{ \tau(\chi^2) \tau(\overline{\chi})}{\tau(\chi)}$, and so the proof is complete.  Finally, when $\chi$ is the Legendre symbol, then $\sum_{u \mymod{p}} \overline{\chi}(u) \chi^2(u+1) = - \chi(-1)$, giving \eqref{eq:FhathiLegendre}.
\end{proof}

\begin{mylemma}[Trivial case]
\label{lemma:FhatchiTrivialPrimeModulusBound}
Suppose that $\chi = \chi_0$ is trivial.  Then we have
\begin{equation}
\label{eq:FhatchiTrivialBound}
\widehat{F}(\chi) = 
\begin{cases}
\frac{1}{p-1}, \qquad &k=1, \\
p^{k/2}, \qquad &k \text{ even}, \\
0, \qquad &k \geq 3 \text{ odd}.
\end{cases}
\end{equation}
%
\end{mylemma}
\begin{proof}
We return to \eqref{eq:Fhatdef2}, and evaluate the sum over $t$ as a Ramanujan sum, giving
\begin{equation}
\widehat{F}(\chi_0) = \frac{1}{\varphi(p^k)} 
\sum_{d|p^k} d \mu(p^k/d)
\sumstar_{\substack{ u \shortmod{p^k} \\ (u  + 1)^2 \equiv 0 \shortmod{d}}}  1
.
\end{equation}
If $k=1$ it is easy to check \eqref{eq:FhatchiTrivialBound} by brute force.
If $k$ is even, then the condition $(u  + 1)^2 \equiv 0 \pmod{d}$, for both $d=p^k$ and $d=p^{k-1}$, is equivalent to $u  + 1 \equiv 0 \pmod{p^{k/2}}$.  The desired formula then follows easily.

Finally, consider $k \geq 3$ odd.  
For $d= p^j$ with $j \in \{k, k-1 \}$,  the condition $(u + 1)^2 \equiv 0 \pmod{d}$ is equivalent to $u + 1 \equiv 0 \pmod{p^{\lceil j/2\rceil}}$.
Thus
\begin{equation}
\label{eq:potato}
\widehat{F}(\chi) = 
\frac{1}{\varphi(p^k)} 
\Big(
p^k \sumstar_{\substack{ u \shortmod{p^k} \\ u + 1 \equiv 0 \shortmod{p^{\frac{k+1}{2}}}}}  1
-
p^{k-1} \sumstar_{\substack{ u \shortmod{p^k} \\ u  + 1 \equiv 0 \shortmod{p^{\frac{k-1}{2}}}}}  1
\Big) = 0. \qedhere
\end{equation}
\end{proof}
From Lemma \ref{lemma:FhatchiTrivialPrimeModulusBound} we easily deduce:
\begin{mycoro}
\label{coro:chitrivialAvgBound}
For each integer $c \geq 1$, let $\chi_0$ denote the trivial character modulo $c$; then
\begin{equation}
\sum_{c \leq x} c^{-1} |\widehat{F}(\chi_0)| \ll x^{\varepsilon}.
\end{equation}
\end{mycoro}


\begin{mydefi}
Suppose $\chi$ is a Dirichlet character of modulus $c$ and conductor $c^*$.  We say $\chi$ is \emph{semi-primitive} if for every prime $p|c$, we have $1 \leq v_p(c^*) < v_p(c)$, where $v_p$ is the $p$-adic valuation.
\end{mydefi}

\begin{mylemma}[Semi-primitive case]
\label{lemma:FhatchiSemiPrimitiveBounds}
Suppose that $\chi$ has conductor $p^j$, with $1 \leq j < k$.  
If $j \not \equiv k \pmod{2}$ then $\widehat{F}(\chi) = 0$.  
If $j \equiv k \pmod{2}$, then
\begin{equation}
\label{eq:FhatSemiPrimitiveBounds}
|\widehat{F}(\chi)| 
\leq 
\begin{cases}
p^{k/2}, \qquad & \chi^2 =1, \\
0, \qquad & \chi^2 \neq 1, \text{ $p$ odd}, \\
0, \qquad &\chi^2 \neq 1, \text{ $p=2$},  k > j+2,  \\
2^{5/2}, \qquad &\chi^2 \neq 1, \text{ $p=2$}, k=j+2.
\end{cases}
\end{equation}
\end{mylemma}

\begin{proof}
We return to the formula \eqref{eq:Fhatdef2}.  
Suppose that $\psi$ of conductor $p^j$ induces $\chi$.
Changing variables $t \rightarrow t + p^j$ keeps $\overline{\chi}(t)$ invariant, showing that the inner sum over $t$ vanishes unless $(u+1)^2 \equiv 0 \pmod{p^{k-j}}$.  The sum over $t$ is hence a Gauss sum repeated $p^{k-j}$ times, giving
\begin{equation}
\widehat{F}(\chi) = \frac{p^{k-j} \tau(\overline{\psi})}{\varphi(p^k)} 
\sum_{\substack{u \shortmod{p^k} \\(u  + 1)^2 \equiv 0 \shortmod{p^{k-j}}}} \overline{\psi}(u) \psi\Big(\frac{(u  + 1)^2}{p^{k-j}}\Big).
\end{equation}
If $k-j$ is odd, then the congruence condition implies $(u  + 1)^2 \equiv 0 \pmod{p^{k-j+1}}$, which causes each summand to vanish.  This proves that $\widehat{F}(\chi) = 0$ if $j \not \equiv k \pmod{2}$.

We proceed under the assumption $j \equiv k \pmod{2}$.  The congruence condition is equivalent to $u  + 1 \equiv 0 \pmod{p^{\frac{k-j}{2}}}$.    We write this as $u = - (1+ p^{\frac{k-j}{2}} y)$, where $y$ now runs modulo $p^{\frac{k+j}{2}}$.  This gives
\begin{equation}
\widehat{F}(\chi) = \frac{p^{k-j} \tau(\overline{\psi})\psi(-1)}{\varphi(p^k)} 
\sum_{\substack{y \shortmod{p^{\frac{k+j}{2}}} }} 
\overline{\psi}(1+ p^{\frac{k-j}{2}} y) 
\psi^2(y).
\end{equation}
Using Fourier inversion in the form $\overline{\psi}(x) = \frac{1}{\tau(\psi)} \sum_{t \mymod{p^j}} \psi(t) e_{p^j}(tx)$, applied with $x=1 + p^{\frac{k-j}{2}} y$, combined with simplifications, gives
\begin{equation}
\sum_{\substack{y \shortmod{p^{\frac{k+j}{2}}} }} 
\overline{\psi}(1+ p^{\frac{k-j}{2}} y) 
\psi^2(y)
=
\frac{\tau(\overline{\psi})}{\tau(\psi)} 
\sum_{\substack{y \shortmod{p^{\frac{k+j}{2}}} }} 
\psi^2(y) e_{p^j}(p^{\frac{k-j}{2}} y).
\end{equation}
Noting the inner sum is the same sum repeated $\frac{p^{\frac{k+j}{2}}}{p^{j}}$ times, we obtain
\begin{equation}
\label{eq:FhatchiSemiPrimitiveSimplified}
\widehat{F}(\chi) = \frac{p^{k-j} \tau(\overline{\psi})^2\psi(-1)}{\varphi(p^k) \tau(\psi)} p^{\frac{k-j}{2}}
\sum_{\substack{y \shortmod{p^{j}} }} 
\psi^2(y)
 e_{p^j}(p^{\frac{k-j}{2}} y)
.
\end{equation}

Now our work breaks into cases.  First suppose that the conductor of $\psi^2$ is $p^j$.  For $p$ odd, this condition is equivalent to $\chi^2 \neq 1$.  In this case, the sum vanishes, giving the second line of \eqref{eq:FhatSemiPrimitiveBounds}.  If $\chi^2 =1$ then $j=1$, and direct evaluation gives a bound consistent with the first line of \eqref{eq:FhatSemiPrimitiveBounds}

Finally, consider $p=2$.  
If $j \in \{2, 3\}$ then $\chi^2 = 1$, and if $j \geq 4$ then the conductor of $\psi^2$ must equal $2^{j-1}$.  Also, note that there are no characters of conductor $2^1$, so this covers everything.  
For $j=2,3$ (equivalently, $\chi^2 = 1$), we bound the inner sum over $y$ trivially, giving the desired bound.  Next suppose $j \geq 4$, so $\psi^2$ has conductor $2^{j-1}$.  If $k > j+2$ then the sum vanishes.  If $k=j+2$, then the inner sum over $y$ is simply a Gauss sum (of modulus $2^{j-1}$) repeated twice. This gives the claimed bound.
\end{proof}

\begin{mycoro}
\label{coro:chiSemiPrimitiveAvgBound}
We have
\begin{equation}
\sum_{c \leq x}  \sum_{\substack{\chi_2 \shortmod{c} \\ \text{semi-primitive}}} 
c^{-1}
|\widehat{F}(\chi_2)| \ll x^{\varepsilon}.
\end{equation}
\end{mycoro}
\begin{proof}
By Lemma \ref{lemma:absFhatIsMultiplicative}, it suffices to show the bound separately for $c$ running over odd intgers, and for $c$ running over powers of $2$.  
For $c$ a power of $2$, we simply use that $|\widehat{F}(\chi)| \ll 1$.

Now consider $c$ odd.
Lemma \ref{lemma:FhatchiSemiPrimitiveBounds} implies that $c$ factors in the form $b_1 b_2^2$ where $b_1$ is square-free, $b_1 | b_2$, and $\chi_2$ is the character of modulus $b_1 b_2^2$ induced by the Jacobi symbol of conductor $b_1$.  Then $|\widehat{F}(\chi_2)| \ll (b_1 b_2^2)^{1/2}$.  To complete the proof, we use
\begin{equation}
\sum_{b_1 b_2^2 \leq x} \sum_{b_1 | b_2} \frac{(b_1 b_2^2)^{1/2}}{b_1 b_2^2} \ll x^{\varepsilon}. \qedhere
\end{equation}
\end{proof}

\section{Completing the proof of Theorem \ref{thm:mainthm}}
\label{section:continue}
We apply Lemma \ref{lemma:DmnFourierMellin} to \eqref{eq:Kformula}, giving
\begin{equation}
\label{eq:Kformula1}
\mathcal{K}(\Delta, T, N)
=
\sum_{X \text{ dyadic}} 
 \sum_{d}
\sum_{g} \frac{\mu(g)}{g} 
 \int_{(\sigma)} \frac{\widetilde{G}_X(-s)}{(dg)^s}
\sum_{c =1}^{\infty} 
\sum_{\chi \shortmod{c}} 
\frac{\widehat{F}(\chi) \chi(\overline{dg})}{c^{1+s}}
B(\chi, s)
 \frac{ds}{2 \pi i}
 ,
\end{equation}
where $\sigma \gg 1$ is large, and where
\begin{equation}
B(\chi, s)
=
\sum_{\substack{d|(m^2, n^2) \\ g|mn/d }} a_m \overline{a_n} \chi(mn) (mn)^s.
\end{equation}
By Lemma \ref{lemma:GMellinproperties}, and shifting the contour far to the right, we may truncate the sum at $cdg \ll C_{\text{max}} := \frac{N^2}{X} (NT)^{\varepsilon}$.  After this truncation all the sums are finite, and we then shift the contour of integration to $\text{Re}(s) = 0$.  Then
\begin{equation}
\label{eq:Kformula2}
|\mathcal{K}(\Delta, T, N)| \ll
\sum_{X \text{ dyadic}} 
\sum_{dg \ll C_{\text{max}}} \frac{1}{g} 
\int_{|t| \ll P T^{\varepsilon} } \frac{\Delta T}{(XP)^{1/2}} 
\sum_{c \ll \frac{C_{\text{max}}}{dg}}
\sum_{\chi \shortmod{c}} \frac{|\widehat{F}(\chi)|}{c} |B(\chi, it)| dt,
\end{equation}
plus a small error term.
Next we factor $c$ and $\chi$ as follows.  Write $\chi = \chi_0 \chi_1 \chi_2$ and $c = c_0 c_1 c_2$, with $(c_i, c_j) = 1$ for $i \neq j$, and where $\chi_j$ has modulus $c_j$.
The factorization is characterized by the assumption that $\chi_0$ is trivial, $\chi_1$ is primitive modulo $c_1$, and that $\chi_2$ is semi-primitive.  This factorization corresponds to the three cases from Lemmas \ref{lemma:FhatchiPrimitiveModulusBound}, \ref{lemma:FhatchiTrivialPrimeModulusBound}, and \ref{lemma:FhatchiSemiPrimitiveBounds}.
Lemma \ref{lemma:absFhatIsMultiplicative} implies $|\widehat{F}(\chi_0 \chi_1 \chi_2)| = |\widehat{F}(\chi_0) \widehat{F}(\chi_1) \widehat{F}(\chi_2)|$.  Using this factorization, and arranging the expression appropriately, we have
\begin{multline}
\label{eq:KboundInTermsOfB}
|\mathcal{K}(\Delta, T, N)| \ll
\sum_{X \text{ dyadic}} 
\sum_{dg \ll C_{\text{max}}} \frac{1}{g} 
\sum_{c_0 c_2 \ll \frac{C_{\text{max}}}{dg}} 
\sum_{\substack{\chi_2 \shortmod{c_2} \\ \text{semi-primitive}}} 
\frac{|\widehat{F}(\chi_0) \widehat{F}(\chi_2)|}{c_0 c_2}
\\
\int_{|t| \ll P T^{\varepsilon} } \frac{\Delta T}{(XP)^{1/2}} 
\sum_{c_1\ll \frac{C_{\text{max}}}{dg c_0 c_2}}
\sumstar_{\chi_1 \shortmod{c_1}} \frac{|\widehat{F}(\chi_1)|}{c_1} |B(\chi_1 \chi_0 \chi_2, it)| dt,
\end{multline}
plus a small error term.

Lemma \ref{lemma:FhatchiPrimitiveModulusBound} implies $|\widehat{F}(\chi_1)| \ll c_1^{\varepsilon}$.  Then by Corollary \ref{coro:LargeSieveVariant} (absorbing $\chi_0 \chi_2$ into the definition of the vector ${\bf a}$), we have that the second line in \eqref{eq:KboundInTermsOfB} satisfies the bound
\begin{multline}
\int_{|t| \ll P T^{\varepsilon} } \frac{\Delta T}{(XP)^{1/2}} 
\sum_{c_1\ll \frac{C_{\text{max}}}{dg c_0 c_2}}
\sumstar_{\chi_1 \shortmod{c_1}} \frac{|\widehat{F}(\chi_1)|}{c_1} |B(\chi_1 \chi_0 \chi_2, it)| dt
\\
\ll (NT)^{\varepsilon} \max_{1 \leq C_1 \ll \frac{C_{\mathrm{max}}}{dg c_0 c_2}} \frac{\Delta T}{C_1 (XP)^{1/2}}
\Big( C_1^2 P + \frac{N}{d'} \Big) \sum_n |a_{nd'}|^2.
\end{multline}
Simplifying and substituting into \eqref{eq:KboundInTermsOfB}, we obtain
\begin{multline}
|\mathcal{K}(\Delta, T, N)| \ll
\sum_{X \text{ dyadic}} 
\sum_{dg \ll C_{\text{max}}} \frac{1}{g} 
\sum_{c_0 c_2 \ll \frac{C_{\text{max}}}{dg}} 
\sum_{\substack{\chi_2 \shortmod{c_2} \\ \text{semi-primitive}}} 
\frac{|\widehat{F}(\chi_0) \widehat{F}(\chi_2)|}{c_0 c_2}
\\
\Delta T
(NT)^{\varepsilon}
\Big(\frac{C_{\text{max}} \sqrt{P}}{dg c_0 c_2 \sqrt{X}} + \frac{N}{d' \sqrt{XP}} \Big) \sum_{n} |a_{nd'}|^2.
\end{multline}
Next we apply Corollaries \ref{coro:chitrivialAvgBound} and \ref{coro:chiSemiPrimitiveAvgBound} to treat the sums over $c_0$ and $c_2$.  We also sum trivially over $g$, which is bounded by a truncated harmonic series.  The sum over $d$ is similar, since $\sum_{d \leq x} \frac{1}{d'} \ll x^{\varepsilon}$.  In all, this gives
\begin{equation}
|\mathcal{K}(\Delta, T, N)| \ll
\sum_{\substack{X \gg \Delta T \\ \text{dyadic}}} 
\Delta T
(NT)^{\varepsilon}
\Big(\frac{C_{\text{max}} \sqrt{P}}{\sqrt{X}} + \frac{N}{\sqrt{XP}} \Big) |{\bf a}|^2
.
\end{equation}
The former term with $\frac{C_{\text{max}} \sqrt{P}}{d \sqrt{X}}$ matches with the first term in \eqref{eq:sketchpartBoundchiPrimitive}; precisely, we have
\begin{equation}
\max_{X \gg \Delta T} \Delta T \frac{N^2}{X} \frac{\sqrt{1 + \frac{T^2}{X}}}{\sqrt{X}} \ll \frac{N^2}{\Delta}.
\end{equation}
For the latter term, we have
\begin{equation}
\max_{X \gg \Delta T} \Delta T \frac{N}{\sqrt{XP}} \ll \Delta N,
\end{equation}
which matches with \eqref{eq:sketchpartchiTrivialCNterm}.  

To summarize, we have shown a bound
\begin{equation}
\mathcal{M}(\Delta, T, N) \ll \Delta T + (NT)^{\varepsilon} \Big(\Delta N + \frac{N^2}{\Delta} \Big)
\end{equation}
which agrees with \eqref{eq:sketchpartPenultimateBound}.  The proof of Theorem \ref{thm:mainthm} is completed by choosing $\Delta$ as described in the paragraph following \eqref{eq:sketchpartPenultimateBound}. \qed

\section{Lower bounds}
\label{section:lowerbound}
\subsection{Proof of Theorem \ref{thm:lowerboundEisPart}}
\label{section:lowerbound1}
We briefly sketch the proof of Theorem \ref{thm:lowerboundEisPart}.  We take $a_n = 1$ for $n=p$ prime, and $a_n = 0$, otherwise.  Note that $\tau_{it}(p^2) = 1 + p^{2it} + p^{-2it}$.  Let
$ A = \sum_p 1$ and $B(t) = \sum_p (p^{2it} + p^{-2it})$, so that
\begin{equation}
\int_{T}^{T+\Delta} 
\Big| \sum_{N \leq n \leq 2N} a_n \tau_{it}(n^2)\Big|^2 dt
=
\int_{T}^{T+\Delta} |A  + B(t)|^2 dt
= \int_T^{T+\Delta} |A |^2 + 2\text{Re}(A  \overline{B}(t)) + |B(t)|^2 dt.
\end{equation}
Note 
\begin{equation}
\label{eq:remarksIntroCtsMainTerm?}
\int_T^{T+\Delta} |A|^2 dt = \int_T^{T+\Delta} \Big|\sum_{p \sim N} a_p\Big|^2 dt  = \Delta \frac{N^2}{(\log{N})^2} (1+o(1)),
\end{equation}
by the prime number theorem.  Meanwhile, by the mean value theorem for Dirichlet polynomials (see Theorem \ref{thm:Gallagher} with $Q=1$), we have
\begin{equation}
\int_{T}^{T+\Delta} |B(t)|^2 dt \ll (\Delta + N) \sum_p 1 \ll (\Delta + N) \frac{N}{\log{N}}.
\end{equation}
Moreover, by Cauchy-Schwarz, we have
\begin{equation}
\int_T^{T+\Delta} |A B(t)| dt \ll \Delta^{1/2} (\Delta + N)^{1/2} \Big(\frac{N}{\log{N}} \Big)^{3/2}.
\end{equation}
Thus
\begin{equation}
\int_{T}^{T+\Delta} 
\Big| \sum_{N \leq n \leq 2N} a_n \tau_{it}(n^2)\Big|^2 dt
= \Delta \frac{N^2}{\log^2{N}} \Big[1+o(1) + O\Big(\frac{(\log{N})^{1/2}}{\Delta^{1/2}}\Big) \Big]. 
\end{equation}
Since $\sum_n |a_n|^2 \sim \frac{N}{\log{N}}$, and using $w_t^{-1} = T^{o(1)}$, this shows Theorem \ref{thm:lowerboundEisPart}.

\subsection{Proof of Proposition \ref{prop:sym2lowerbound}}
The idea and details of the proof are similar to those presented in Section \ref{section:lowerbound1}, so we will be brief.
We take $a_n$ defined by $a_{n} = 1$ for $n=p^2 \asymp N$, with $p$ prime, and $a_n = 0$ otherwise.  Note $\lambda_{\mathrm{sym}^2 u_j}(p^2) = 1 + \lambda_j(p^2)$.  Let $A = \sum_{p} a_{p^2} \asymp \frac{\sqrt{N}}{\log{N}}$ and $B_j = \sum_{p} \lambda_j(p^2)$.  Then we have 
\begin{equation}
\sum_{T \leq t_j \leq T + \Delta} w_j^{-1} |A + B_j|^2 
=
\sum_{T \leq t_j \leq T + \Delta} w_j^{-1} \Big(|A|^2 + 2 \text{Re}(A \overline{B_j}) + |B_j|^2 \Big).
\end{equation}
The rest of the proof now plays out nearly identically to that of Theorem \ref{thm:lowerboundEisPart}.  The conditions in place in Proposition \ref{prop:sym2lowerbound}, together with \eqref{eq:Sym2TrivialBound}, show that $\sum_{t_j} w_j^{-1} |B_j|^2 = o(\sum_{t_j} w_j^{-1} |A|^2)$, which implies that
\begin{equation}
\sum_{T \leq t_j \leq T+\Delta} w_j^{-1} |A+B_j|^2 \gg \Delta T |A|^2
\asymp \Delta T \frac{\sqrt{N}}{\log{N}} |{\bf a}|^2. \qedhere
\end{equation}

\end{document}